\documentclass{amsart}

\usepackage{amsthm}

\def\GL{\operatorname{GL}}
\def\SL{\operatorname{SL}}
\def\O{\operatorname{O}}
\def\Vol{\operatorname{Vol}}
\def\Div{\operatorname{div}}
\def\Ric{\operatorname{Ric}}
\def\trace{\operatorname{trace}}

\def\R{\mathbb{R}}
\def\cN{\mathcal{N}}

\newtheorem{prop}{Proposition}
\newtheorem{thm}{Theorem}
\newtheorem{cor}{Corollary}
 \theoremstyle{remark}
\newtheorem{rem}{Remark}

\begin{document}

\author{S.~E.~Stepanov}
 \address{S.~E.~Stepanov: Financial University under the Government
  of the Russian Federation, Moscow, Russia}
 \email{s.e.stepanov@mail.ru}

\title{On Geometric Analysis of the Dynamics of Volumetric Expansion}

\begin{abstract}
In this paper, we discuss the global aspect of the geometric dynamics of
volumetric expansion and its application to the problem of the existence in
the space-time of compact and complete spacelike hypersurface.
\end{abstract}



\maketitle

\section{Introduction}\label{sec0}

In this paper, we discuss the global aspect of the dynamics of volumetric
expansion and its connection with the problem on the existence of spacelike
closed and complete hypersurfaces in the space-time. This problem is
interesting for several reasons. First, manifolds that admit global spacelike
hypersurfaces are more relevant to physics than manifolds in which such
hypersurfaces cannot exist (see~\cite[p.~207]{1}. Second, the problem on the
existence of closed spacelike hypersurfaces (see~\cite{2}, \cite[p.~164]{3})
is closely related to the accelerated expansion of the Universe
(see~\cite{4}).

In Sec.~\ref{sec1}, we recall basic definitions and elementary facts of the
dynamics of volumetric expansion in the case of the absence of a metric; here
we use the methods proposed in~\cite{5}. In Sec.~\ref{sec2}, we examine this
problem for complete Riemannian manifolds by the methods developed
in~\cite{6, 7}. Finally, in Secs.~\ref{sec3} and~\ref{sec4} we discuss
applications of the dynamics of volumetric expansion to the problem on the
existence of compact, complete, spacelike hypersurfaces in Lorentzian
manifolds and in the space-time.

\section{Dynamics of Volumetric Expansion without Metric}\label{sec1}

Let $M$ be a differentiable manifold of dimension~$n$ ($n\ge 2$) and $L(M)$
be the bundle of linear frames over~$M$ with the structure group~$\GL(n,\R)$.
We define an $\SL(n,\R)$-structure on the manifold~$M$ as an
$\SL(n,\R)$-subbundle of the bundle~$L(M)$ for the subgroup~$\SL(n,\R)$
of~$\GL(n,\R)$. It is well known (see~\cite[p.~13]{8}) that the
$\SL(n,\R)$-structure on~$M$ is nothing but the volume element~$\omega$,
i.e., an $n$-form $\omega$, which does not vanish everywhere on~$M$.
Moreover, it is known that a connected manifold~$M$ admits an
$\SL(n,\R)$-structure if and only if it is orientable. Hence, we assume in
the sequel that the $n$-form $\omega$ satisfies the condition
$$
 \omega \left(\frac{\partial}{\partial x^{1}},\dots,
 \frac{\partial}{\partial x^{n}}\right)>0
$$
for an arbitrary local coordinate system $x^{1},\dots,x^{n}$ matched with the
orientation of~$M$ (see~\cite[p.~259]{9}, \cite[p.~86]{10}).

The volume element $\omega$ on the manifold~$M$ allows one to introduce the
integral $\int_M f:=\int_M f\omega$ of an arbitrary compactly supported
in~$M$, differentiable function~${f\colon M\to \R}$ (see~\cite[p.~87]{10}).
In particular, if the manifold~$M$ is closed (i.e., compact without
boundary), then we can define its \textit{volume} as follows
(see~\cite[p.~87]{10}):
$$
 \Vol_{\omega}(M):=\int_M 1= \int_M \omega >0.
$$

Let $\xi$ be a differentiable vector field on~$M$. It is well known
(see~\cite[pp.~22-23]{9}, \cite[p.~140]{11}, \cite[pp.~27, 29]{12}) that in a
neighborhood~$U$ of each point of the manifold~$M$ the field~$\xi$ generates
a  \textit{local flow}, which is a local one-parameter group of infinitesimal
diffeomorphisms or, in other words, transformations $\phi_{t}(x):U\to M$.
Such transformations are given by the formula
$$
 \phi_{t}(x^{k})=\bar{x}^{k}=x^{k}+t\xi^{k}
$$
in an arbitrary local coordinate system $x^{1},\dots,x^{n}$ in a
neighborhood~$U$, where $t\in (-\varepsilon ,+\varepsilon)\subset \R$ is a
parameter and $\xi =\xi ^{k}\partial_{k}$. The converse assertions is also
valid (see~\cite[pp.~21-22]{9}, \cite[p.~140]{11}), namely, a local flow (or,
in other words, a local one-parameter group of infinitesimal transformations
of the manifold~$M$ consisting of diffeomorphisms $\phi_{t}(x):U\to M$), for
some open set $U \subset M$, an interval $(-\varepsilon,+\varepsilon) \subset
\R$, and arbitrary $t\in(-\varepsilon,+\varepsilon)$ and $x\in U$, induces a
vector field~$\xi$ on~$U$ as follows. At each point $x\in U$, we define a
vector~$\xi_{x}$ tangent to the curve $x(t)=\phi_{t}(x)$ and such that
$\xi^{k}=dx^{i}/dt$ for $k=1,\dots,n$ in a local coordinate system
$x^{1},\dots,x^{n}$ in~$U$. The curve $x(t)=\phi_{t}(x)$ is called the
\textit{trajectory} of the flow. If the transformation group is
\textit{global}, then the vector field $\xi$ generated by it is said to be
\textit{complete} (see~\cite[p.~22]{9}, \cite[p.~29]{12}). In particular, on
a compact manifold~$M$ each differentiable vector field~$\xi $ is complete
(see~\cite[p.~23]{9}).

The vector field $\xi$ is also called the \textit{velocity vector} (field) of
the flow. For an arbitrary differentiable tensor field~$T$, one can consider
its \textit{Lie derivative} along trajectories of the flow with the velocity
vector~$\xi$:
$$
 L_{\xi}T:=\frac{d}{dt}\left(\phi_{t}^{\ast}T\right)_{|t=0}
$$
(see~\cite[p.~36]{9}, \cite[p.~71]{10}). It is well known
(see~\cite[p.~211]{13}) that the Lie derivative $L_{\xi}T$ measures the
\textit{rate of the change of the tensor~$T$} under deformations determined
by the one-parameter group of differentiable transformations~$\phi_{t}$
generated by the vector field~$\xi$. In particular, the rate $L_{\xi}\omega$
of the change of the volume element $\omega$ or, in other words, the
\textit{rate of volumetric expansion} under deformations determined by the
one-parameter group of differentiable transformations~$\phi_{t}$ generated by
the field~$\xi$ can be calculated by the formula (see~\cite[p.~259]{9},
\cite[p.~212]{13})
\begin{equation}\label{eq1}
 L_{\xi}\omega :=\big(\Div_{\omega}\xi \big)\cdot \omega.
\end{equation}
Due to~\eqref{eq1}, the function $\Div_{\omega} \xi$ is called the
\textit{logarithmic rate of change of the volume} (or the \textit{rate of
volumetric expansion}) along the flow generated by the field~$\xi$
(see~\cite[p.~195]{12}). On the other hand, for a vector field $\xi$ with
compact support in~$M$, \textit{Green's theorem} is valid
(see~\cite[p.~259]{9}):
\begin{equation}\label{eq2}
 \int_M \big(\Div_{\omega} \xi \big) \cdot \omega =0.
\end{equation}
Obviously, the conditions $\Div_{\omega} \xi>0$ and $\Div_{\omega}\xi<0$
contradict~\eqref{eq2}. If $\Div_{\omega}\xi \ge 0$ or $\Div_{\omega}\xi \le
0$, then Eq.~\eqref{eq2} implies that $\Div_{\omega}\xi=0$. This means that
$L_{\xi}\omega=0$, i.e., the one-parameter group of differentiable
transformations $\phi_{t}$ leaves $\omega $ invariant and the vector
field~$\xi$ is an \textit{infinitesimal automorphism} of the
$\SL(n,\R)$-structure (see~\cite[pp.~9-10]{8}). In hydrodynamics (even in the
absence of a metric), such a vector field~$\xi$ is said to be
\textit{divergence-free} and the flow generated by it is said to be
\textit{incompressible} (see~\cite[p.~125]{14}). The following assertion is
obvious.

\begin{prop}\label{prop1}
Let $(M,\omega)$ be a connected differentiable manifold equipped with an
$\SL(n,\R)$-structure and let a flow with a compactly supported velocity
vector~$\xi$ be given on~$M$. The volume element $\omega$ cannot increase
(decrease) along trajectories of the flow. If the volume element~$\omega$ is
a nondecreasing (or nonincreasing) function along trajectories of the flow,
then this flow is incompressible and its rate of volumetric expansion is
equal to zero.
\end{prop}

\begin{proof}
At each point of the manifold~$M$, the signs of the functions of the
logarithmic rate $\Div_{\omega}\xi$ and the rate $L_{\xi}\omega$ of
volumetric expansion coincide. Moreover, $\Div_{\omega}\xi=0$ if and only if
$L_{\xi}\omega=0$. This means that of the velocity vector~$\xi$ has a compact
support in the manifold~$M$ with the volume element~$\omega$, then the
assertion of the theorem is a direct consequence of Green's theorem.
\end{proof}

Since the Lie derivative $L_{\xi }\omega$ of the volume element $\omega$
measures the rate of its change under the action of the group of
differentiable transformations~$\phi_{t}$ generated by the field~$\xi$, the
Lie derivative~$L_{\xi}(L_{\xi}\omega)$, in its turn, measures the
\textit{acceleration} of volumetric expansion, i.e., the
\textit{acceleration} of the change of the volume element~$\omega$ along
trajectories of the flow with the velocity vector~$\xi$. In this case, the
following relation holds:
\begin{multline}\label{eq3}
 L_{\xi}(L_{\xi}\omega)
 =L_{\xi}\big((\Div_{\omega}\xi)\cdot\omega\big)
 =\big(L_{\xi}(\Div_{\omega}\xi)\big)\cdot \omega
 +(\Div_{\omega}\xi)\cdot L_{\xi} \omega
 \\
 =\Big(L_{\xi}(\Div_{\omega}\xi)+(\Div_{\omega}\xi)^{2}\Big)\cdot \omega,
\end{multline}
where the function $L_{\xi}(\Div_{\omega}\xi)$ characterizes the rate of the
change of the logarithmic rate of volumetric expansion $\Div_{\omega}\xi$
along trajectories of the flow generated by the field~$\xi$. Obviously, the
vanishing of the acceleration of volumetric expansion
$L_{\xi}(L_{\xi}\omega)$ leads to the condition
$L_{\xi}(\Div_{\omega}\xi)=-(\Div_{\omega}\xi)^{2}\le 0$. As a result, the
logarithmic rate of volumetric expansion either decreases or vanishes along
trajectories of the flow. On the other hand, if the logarithmic rate of
volumetric expansion is a nondecreasing (or even zero) function along
trajectories of the flow, then the rate of volumetric expansion of this flow
is also a nondecreasing function. Obviously, the growth condition
$L_{\xi}(\Div_{\omega}\xi)>0$ for the logarithmic rate implies the increasing
of the rate of volumetric expansion since in this case
$L_{\xi}(L_{\xi}\omega)>0$.

The vector field $(\Div_{\omega}\xi)\xi$ is called the \textit{vector of
logarithmic rate of volumetric expansion}. Based on Eq.~\eqref{eq3}, we can
prove the following assertion.

\begin{prop}\label{prop2}
Let $(M,\omega)$ be a connected differentiable manifold equipped with an
$\SL(n,\R)$-structure and let a flow with velocity vector~$\xi$ be given
on~$M$ such that the vector of logarithmic rate of volumetric expansion
$(\Div_{\xi}\xi)\xi $ has a compact support in~$M$. The rate of volumetric
expansion of this flow cannot increase (decrease) along trajectories. If the
rate $L_{\xi}\omega$ of volumetric expansion along trajectories of the flow
is a nondecreasing (or nonincreasing) function, then this flow has a constant
rate of volumetric expansion. In particular, if the logarithmic rate of
volumetric expansion $\Div_{\omega}\xi$ is a nondecreasing function along
trajectories of the flow, then the flow is incompressible and its rate of
volumetric expansion~$L_{\xi}\omega$ is equal to zero.
\end{prop}

\begin{proof}
Let $M$ be a differentiable manifold with volume element~$\omega$. First, we
note that for the vector field $(\Div_{\omega}\xi)\xi$, which is compactly
supported in~$M$, Green's theorem has the form
\begin{equation}\label{eq4}
 \int_M \Div \big( (\Div_{\omega}\xi)\xi \big)\cdot \omega
 =\int_M \Big( L_{\xi}(\Div_{\omega}\xi)+(\Div_{\omega}\xi)^{2} \Big)
 \cdot \omega =0.
\end{equation}
Second, it is easy to see from~\eqref{eq3} that the inequality
$$
 L_{\xi}(L_{\xi}\omega)>0 \qquad(L_{\xi}(L_{\xi}\omega)<0),
$$
which is valid everywhere on~$M$, implies the inequality
$$
 L_{\xi}(\Div_{\omega}\xi)+(\Div_{\omega}\xi)^{2}>0
 \qquad (L_{\xi}(\Div_{\omega}\xi)+(\Div_{\omega}\xi)^{2}<0),
$$
which contradicts~\eqref{eq4}. Similarly, from~\eqref{eq3} we also conclude
that the inequality
$$
 L_{\xi}(L_{\xi}\omega)\ge0 \qquad
 (L_{\xi}(L_{\xi}\omega)\le0),
$$
which is valid everywhere on~$M$, implies the inequality
$$
 L_{\xi}(\Div_{\omega}\xi)+(\Div_{\omega}\xi)^{2}\ge0
 \qquad(L_{\xi}(\Div_{\omega}\xi)+(\Div_{\omega}\xi)^{2}\le0).
$$
Therefore, Eq.~\eqref{eq4} implies that
$$
 L_{\xi}(\Div_{\omega}\xi)+(\Div_{\omega}\xi)^{2}=0
$$
and hence
$$
 L_{\xi}(L_{\xi}\omega)=0.
$$
In particular, from~\eqref{eq4} for $L_{\xi}(\Div_{\omega}\xi)\ge0$ we
conclude that $\Div_{\omega}\xi=0$ and hence the equality $L_{\xi}\omega=0$
holds.
\end{proof}

The following assertion is valid.

\begin{prop}\label{prop3}
On a closed (i.e., compact without boundary) differentiable manifold
$(M,\omega)$ equipped with an $\SL(n,\R)$-structure, there are no flows with
velocity vector~$\xi$ and the nondecreasing along the flow logarithmic rate
of volumetric expansion $\Div_{\omega}\xi$, if the acceleration
$L_{\xi}(\Div_{\omega}\xi)>0$ at least at one point.
\end{prop}

This assertion is valid since the conditions imposed on the scalar functions
$\Div_{\omega}\xi$ and $L_{\xi}(\Div_{\omega}\xi)$ contradict
Eq.~\eqref{eq4}, which is also valid on any closed (i.e., compact without
boundary) manifold~$M$.

\section{Dynamics of Volumetric Expansion on Complete
 Riemannian Manifolds}\label{sec2}

Let $\O(n,\R)$ be a subgroup of orthogonal transformations from~$\GL(n,\R)$.
We determine an $\O(n,\R)$-structure on an $n$-dimensional ($n\ge 2$)
differentiable manifold~$M$ as an $\O(n,\R)$-subbundle of the bundle~$L(M)$.
It is well known (see~\cite[p.~13]{8}) that an $\O(n,\R)$-structure on~$M$ is
nothing but the Riemannian metric~$g$ on~$M$. In this casem the pair $(M,g)$
is called a Riemannian manifold. For a Riemannian manifold $(M,g)$, we can
introduce the canonical volume element~$\omega=dv$, which is defined in an
arbitrary oriented local coordinate system $x^{1},\dots,x^{n}$ on~$M$ by the
formula $dv=\sqrt{\det g} dx^{1}\wedge \dots\wedge dx^{n}$.

Recall that a Riemannian manifold $(M,g)$ is complete if each geodesic
in~$(M,g)$ can be extended for arbitrarily large values of its canonical
parameter (see~\cite[p.~166]{9}). It is known that on each connected
differentiable manifold there exists a structure of a complete Riemannian
manifold (see~\cite[p.~175]{15}).

On a complete Riemannian manifold, the \textit{generalized Green's theorem}
holds (see~\cite{16, 17}). Namely, for a differentiable vector field~$X$ on a
complete (noncompact) oriented Riemannian manifold~$M$, the conditions $\Div
X \ge 0$ (or $\Div X \le 0$) and $|X|\in L^{1}(M,g)$ imply $\Div X=0$; here
$|X|=\sqrt{g(X,X)}$ and $\Div X=\Div_{\Omega}X$ for $\Omega=dv$. Taking into
account this theorem, we can reformulate Proposition~\ref{prop1} as follows.

\begin{cor}\label{cor1}
Let on a complete noncompact oriented Riemannian manifold $(M,g)$, a flow
with complete field of velocity vectors $\xi$ be given such that $|\xi|\in
L^{1}(M,g)$. If the volume element of the manifold is a nondecreasing (or
nonincreasing) function along trajectories of the flow, then the flow is
incompressible and its rate of volumetric expansion is equal to zero.
\end{cor}

If we take the vector of logarithmic rate $(\Div\xi)\xi$ as~$X$ in the
generalized Green's theorem, then Proposition~\ref{prop2} implies the
following assertion.

\begin{cor}\label{cor2}
Let on a complete, noncompact, oriented Riemannian manifold $(M,g)$, a flow
with complete field of velocity vectors~$\xi$ be given such that
$|(\Div\xi)\xi|\in L^{1}(M,g)$. If the rate of volumetric expansion is a
nondecreasing (or nonincreasing) function along trajectories of the flow,
then this flow has a constant rate of volumetric expansion. In particular, if
the logarithmic rate of volumetric expansion is a nondecreasing function
along trajectories of the flow, then the flow is incompressible and its rate
of volumetric expansion is equal to zero.
\end{cor}

As an example, we consider a local one-parameter group of infinitesimal
transformations (i.e., a~flow) generated by the vector field~$\xi$ of a Ricci
soliton, i.e., a Riemannian manifold $(M,g)$ such that
\begin{equation*}
 -2\Ric=L_{\xi} g+2\lambda g,
\end{equation*}
where $\Ric$ is the Ricci tensor of the metric~$g$, $L_{\xi}g$ is the
derivative of the metric~$g$ in the direction of the field~$\xi$,
and~$\lambda$ is a constant (see~\cite[p.~22]{18}). For the vector of
logarithmic rate of volumetric expansion $(\Div\xi)\xi$, we have
(see~\cite{19})
\begin{equation}\label{eq5}
 \Div(\Div\xi)\xi =-L_{\xi} s+(s+n\lambda)^{2},
\end{equation}
where $s =\trace_{g} \Ric$ is the scalar curvature of the metric~$g$.

\begin{prop}\label{prop4}
Let $(M,g)$ be a complete, noncompact, oriented Ricci soliton with a complete
vector field~$\xi$ such that $|(\Div\xi)\xi|\in L^{1}(M,g)$. If the scalar
curvature~$s$ of the Ricci soliton is a nonincreasing function along
trajectories of the flow generated by the field~$\xi$, then $s=-n\lambda$ and
the flow is incompressible.
\end{prop}

\section{Dynamics of Volumetric Expansion
 on Lorentzian Manifolds}\label{sec3}

On an $n$-dimensional ($n\ge3$) differentiable manifold~$M$, we introduce an
$\O(1,n-1)$-structure as an $\O(1,n-1)$-subbundle of the bundle~$L(M)$, where
$\O(1,n-1)$ is the subgroup of Lorentz transformations in~$\GL(n,\R)$. It is
well known  (see~\cite[p.~13]{8}) that an $\O(1,n-1)$-structure on~$M$ is
nothing but a pseudo-Riemannian metric~$g$ on~$M$ with Lorentz signature
$({-}{+}{+}\dots{+}$. It is also known that any noncompact differentiable
manifold admits an $\O(1,n-1)$-structure, i.e., possesses a metric with
Lorentz signature. An $\O(1,n-1)$-structure on a compact manifold exists if
and only if its Euler characteristics is equal to zero
(see~\cite[p.~50]{20}). A~pair $(M,g)$ is called a \textit{Lorentzian
manifold} (see~\cite[p.~50]{20}). The canonical volume element on~$(M,g)$ has
the form $dv=\sqrt{|\det g|} dx^{1}\wedge \dots\wedge dx^{n}$ for an
arbitrary oriented local coordinate system $x^{1},\dots,x^{n}$ on~$M$.

In on an $n$-dimensional ($n\ge3$) Lorentzian manifold $(M,g)$ a complete
unit timelike vector field $\xi$ is given, then~$(M, g)$ is said to be
\textit{time oriented by the field~$\xi$} (see~\cite[p.~50]{20}). In this
case, the following consequence of Proposition~\ref{prop2} is valid.

\begin{cor}\label{cor3}
Assume that on an $n$-dimensional ($n\ge3$) Lorentzian manifold $(M, g)$
there exists a flow with a unit timelike compactly supported velocity
vector~$\xi$. The rate of volumetric expansion of this flow cannot be a
function that increases (descreases) along its trajectories and, in
particular, the logarithmic rate of volumetric expansion cannot increase.
\end{cor}

Assume that in~$(M, g)$ there exists a spacelike hypersurface~$M'$
(see~\cite[p.~92]{20}) orthogonal to trajectories of the flow generated by
the vector field~$\xi$. In this case, the \textit{acceleration vector} of the
flow $\dot{\xi}=\nabla_{\xi}\xi$ is a tangent vector for~$M'$. At points
of~$M'$, the following relations holds (see~\cite{4, 21}):
\begin{equation}\label{eq6}
 \Div\dot{\xi}=\Ric(\xi,\xi)+g(\sigma,\sigma)
 +(n-1)^{-1} (\Div\xi)^{2}+L_{\xi}(\Div\xi),
\end{equation}
where $g(\sigma,\sigma)\ge 0$ is the square of the traceless part of the
\textit{second fundamental form} of the hypersurface~$M'$
(see~\cite[p.~93]{20}). The following theorem is valid.

\begin{thm}\label{thm1}
Assume that in an $n$-dimensional ($n\ge3$) Lorentzian manifold $(M,g)$,
there exist a spacelike hypersurface~$M'$ and a flow orthogonal to it with
unit timelike velocity vector~$\xi$. If $\Ric(\xi,\xi)\ge0$ and for the
logarithmic rate of volumetric expansion $\Div\xi$ we have
$L_{\xi}(\Div\xi)\ge 0$ at each point of~$M'$ and, moreover, at least at one
point $L_{\xi}(\Div\xi)>0$, then the hypersurface~$M'$ cannot be closed.
\end{thm}

\begin{proof}
Assume that the hypersurface~$M'$ is a closed (i.e., compact without
boundary) spacelike submanifold in~$(M,g)$, which is orthogonal at each point
to the unit timelike vector field~$\xi$. Applying Green's theorem to the
vector field $\dot{\xi}=\nabla_{\xi}\xi$, which is tangent to~$M'$, we obtain
the integral relation
\begin{equation}\label{eq7}
 \int_M \Big( \Ric(\xi,\xi)+g(\sigma,\sigma)+(n-1)^{-1}(\Div\xi)^{2}
 +L_{\xi}(\Div\xi) \Big) \cdot dv'=0,
\end{equation}
where $dv'$ is the volume element of~$M'$. Obviously, the conditions
$\Ric(\xi,\xi)\ge0$ and $L_{\xi}(\Div\xi)\ge0$ everywhere on~$M'$ and
$L_{\xi}(\Div\xi)>0$ at least at one point of~$M'$ contradict the integral
relation~\eqref{eq7}.
\end{proof}

\begin{rem}
We recall the following classical result (see~\cite{2}): in an
$n$-dimensional ($n\ge3$) compact space-time with the Ricci tensor satisfying
the condition $\Ric(X,X)>0$ for all nonspacetime vectors $X\ne0$, there are
no closed (i.e., compact without boundary) spacelike hypersurfaces.
\end{rem}

If we assume that the spacelike hypersurface~$M'$ is a complete (noncompact)
oriented Riemannian manifold, then using the generalized Green's theorem we
can prove the following.

\begin{thm}\label{thm2}
Assume that in an $n$-dimensional ($n\ge3$) Lorentzian manifold $(M,g)$ there
exist a spacelike complete (noncompact) oriented hypersurface~$M'$ and a flow
orthogonal to it with unit timelike velocity vector~$\xi$ such that
$|\dot{\xi}|\in L^{1}(M,g)$. If $\Ric(\xi,\xi)\ge 0$ and the logarithmic rate
of volumetric expansion $\Div\xi$ satisfies the inequality
$L_{\xi}(\Div\xi)\ge0$ at each point of~$M'$, then $M'$ is a completely
geodesic submanifold of~$(M,g)$.
\end{thm}

\begin{proof}
By the condition of the theorem, the vector field $\dot{\xi}=\nabla_{\xi}\xi$
is tangent for the complete, noncompact, oriented Riemannian manifold~$M'$ on
which Eq.~\eqref{eq5} holds. If we impose the conditions $\Ric(\xi,\xi)\ge 0$
and $L_{\xi}(\Div\xi)\ge0$ everywhere on~$M'$, then due to~\eqref{eq5} the
inequality $\Div\dot{\xi}\ge 0$ also holds. In this case, by the generalized
Green's theorem we conclude that $\Div\dot{\xi}=0$. Then, in particular,
from~\eqref{eq5} we see that $\Div\xi=\sigma=0$. This means (see~\cite{4,
21}) that the second fundamental form of the hypersurface~$M'$ vanishes and
hence the hypersurface~$M'$ itself is a \textit{completely geodesic
submanifold of~$(M,g)$} (see~\cite[pp.~93-94]{20}).
\end{proof}

\section{Dynamics of Volumetric Expansion in Space-Time}\label{sec4}

Recall that a \textit{space-time} is a connected four-dimensional oriented
Lorentzian manifold $(M,g)$ (see~\cite[c.~27]{22}). For the case $n=4$,
Eq.~\eqref{eq5} follows from the Landau--Raychaudhuri equation
(see~\cite[pp.~97--98]{22}), which describes the dynamics of flows of the
cosmological liquid in the space-time. Streamlines of this liquid are
trajectories of the flow generated by a unit timelike vector field~$\xi$
(see~\cite[p.~47]{20} and~\cite[p.~92]{23}). The hydrodynamical sense of
variables appearing in Eq.~\eqref{eq5} is as follows (see~\cite[p.~219]{1}
and~\cite[p.~96]{22}): $\sigma$ is the tensor of transversal shear, $\theta$
is the volumetric divergence, and $\dot{\xi}=\nabla_{\xi}\xi$ is the vector
field of the 4-acceleration of the cosmological liquid. In the case of the
\textit{perfect liquid} we have $\Ric(\xi,\xi)=4\pi (\mu+3\rho)$, where $\mu$
i the energy density and $\rho$ is the pressure (see~\cite[p.~98]{22}).
Moreover, in the case $n=4$, following Hawking and Penrose
(see~\cite[p.~539]{24}), the inequality $\Ric(X,X)\ge 0$, which is valid for
all unit timelike vectors~$X$, is called the \textit{energy condition} for
the space-time.

We also recall that the increasing of the logarithmic rate of volumetric
expansion implies the increasing of the rate of volumetric expansion; this is
directly related to the problem on the accelerated expansion of the Universe.
In the case of a four-dimensional space-time, the last two theorems acquire a
physical content. The following two assertions are valid (cf.~\cite{2}
and~\cite[p.~164]{3}).

\begin{cor}\label{cor4}
Assume that in a space-time satisfying the energy condition, there exist a
spacelike hypersurface and a flow of the cosmological liquid orthogonal to
it. If the logarithmic rate of volumetric expansion is a nondecreasing
function along streamlines and on the hypersurface there exists growth
points, then the hypersurface cannot be closed.
\end{cor}

\begin{cor}
Assume that in a four-dimensional space-time $(M,g)$ satisfying the energy
condition, there exists a spacelike complete oriented hypersurface and a flow
of the cosmological liquid orthogonal to it with velocity vector~$\xi$ such
that $\big|\dot{\xi}\big|\in L^{1}(M,g)$. If the logarithmic rate of
volumetric expansion along streamlines is a nondecreasing (or even zero)
function, then the hypersurface is completely geodesic.
\end{cor}

Each closed, oriented three-dimensional manifold~$M'$ is the boundary of a
certain four-dimensional manifold (see~\cite{25, 26}). In our case, this
means that for a closed, oriented, spacelike hypersurface~$M'$, there exists
a four-dimensional submanifold $N \subset M$ such that $\partial N=M'$. Since
the space-time $(M,g)$ is oriented, the submanifold~$N$ is also oriented. In
this case, we can assume that the hypersurface~$M'$ is also oriented and its
orientation is induced by an imaginary unit normal vector~$\cN$ directed
outward at each point $x\in M'$. In this case, the divergence theorem
(see~\cite{27}) has the form
\begin{equation*}
 \int_N (\Div\xi)\,dv= -\int_{M'} g(\xi,\cN)\,dv'
 =\int_{M'} dv' =\Vol(M')>0,
\end{equation*}
where $\xi=\cN$ at each point $x\in M'$. In its turn, the divergence theorem
for the vector field $(\Div\xi)\xi$ has the form
\begin{equation}\label{eq8}
 \int_N \Big(L_{\xi}(\Div\xi)+(\Div\xi)^{2}\Big)\,dv
 =\int_{M'} (\Div\xi)\,dv'.
\end{equation}
Now it is easy to see that the conditions $L_{\xi}(\Div\xi)\ge 0$ at all
points of~$N$ and $\Div\xi=0$ at all points of~$M'$ contradict
Eq.~\eqref{eq8}. The following theorem is valid.

\begin{thm}\label{thm3}
Assume that in the space-time there exist a closed spacelike hypersurface and
an outward flow of the cosmological liquid orthogonal to it. This flow cannot
simultaneously satisfy the following two conditions:
\begin{enumerate}
 \item 
along streamlines on the submanifold whose boundary coincides with this
hypersurface, the acceleration of volumetric expansion (or the logarithmic
rate of volumetric expansion) is a nondecreasing function and there exist
points at which it is nonzero;
 \item 
on the hypersurface, the logarithmic rate of volumetric expansion vanishes.
\end{enumerate}
\end{thm}

\subsection*{Acknowledgment.}
The author was partially supported by the Russian Foundation for Basec
Research (project Nos.16-01-00756-a and 16-01-00053-a). The author would 
like to thank his friend Prof. A.V. Ovchinnikov for his translation into 
English language of the paper and for  his editorial suggestions, which 
led to an improvement of this paper.

\end{document}